\documentclass[11pt,a4paper, final]{article}
\usepackage{amsmath}
\usepackage{amsthm}
\usepackage{amssymb}
\usepackage{graphicx}
\usepackage{setspace}
\usepackage[utf8]{inputenc}

\def\RR{\mathbb{R}}

\def\ltord{L^2(I; \mathbb{R} ^d)}
\def\GG{\mathcal{G}}

\def\ZZ{\mathbb{Z}}
\def\SS{\mathbb{S}}

\def\TT{\mathbb{T}}
\def\Tt{\mathrm{T}}

\def\eps{\varepsilon}
\def\DD{\mathcal{D}}
\def\Pp{\mathcal{P}}

\def\ltordz{L^2 (I; \ZZ ^d)}
\def\xd{\dot{x}}

\def\to{\rightarrow}
\def\tx{\textbf{x}}
\def\ty{\textbf{y}}

\def\tp{\textbf{p}}
\def\tq{\textbf{q}}

\def\dw{\text{dist}_{weak}}

\numberwithin{equation}{section}
\newtheorem*{thm5.1}{Theorem 5.1}
\newtheorem*{thm5.2}{Theorem 5.2}
\newtheorem*{thm4.3}{Theorem 4.3}
\newtheorem{theorem}{Theorem}[section]
\newtheorem{lemma}{Lemma}[section]
\newtheorem{proposition}{Proposition}[section]
\newtheorem{corollary}{Corollary}[section]
\newtheorem{definition}{Definition}[section]
\newtheorem{remark}{Remark}[section]

\usepackage{srcltx}
\begin{document}

\centerline{\textbf{An infinite-dimensional Weak KAM theory via random variables}}

\vskip .5 cm

\centerline{Diogo Gomes\footnote{King Abdullah University of Science and Technology (KAUST), CEMSE Division , Thuwal 23955-6900, Saudi Arabia.\\
email: diogo.gomes@kaust.edu.sa.
}}
\centerline{and}
\centerline{Levon Nurbekyan\footnote{King Abdullah University of Science and Technology (KAUST), CEMSE Division , Thuwal 23955-6900, Saudi Arabia.\\
		email: levon.nurbekyan@kaust.edu.sa}}

\vskip 0.5 cm

\centerline{Abstract}

\vskip 0.5 cm


We develop several aspects of the infinite-dimensional Weak KAM theory using a random variables' approach. We prove that the infinite-dimensional cell problem admits a viscosity solution that is a fixed point of the Lax-Oleinik semigroup. Furthermore, we show the existence of invariant minimizing measures and calibrated curves defined on $\RR$.

\vskip .5 cm

\leftline{Keywords: dynamical systems, weak KAM theory, calculus of variations}
\leftline{MSC2010 number: 49J99, 49L25, 49L20, 37K99}


D. Gomes was partially supported by KAUST baseline and start-up funds and KAUST SRI, Center for Uncertainty Quantification in Computational Science and Engineering. 

\section{Introduction}

In this paper, we study dynamical systems with an infinite number indistinguishable particles on a $d$-dimensional torus $\Tt^d$. We extend and apply methods from the \textit{Weak KAM theory} \cite{F,fathi'11,fathi'97a,fathi'97b,fathi'98,fathigiuliani'09,fathimaderna'07,fathisiconolfi'04,fathisiconolfi'05,bernard'07,bernard'12,bernardbuffoni'07,gomes'00,gomes'02,gomes'03a,gomes'03b} to the infinite-dimensional setting using a random variables' approach. In particular, we prove that the \textit{infinite-dimensional cell problem} admits a viscosity solution that is a fixed point of the Lax-Oleinik semigroup. Furthermore, we construct \textit{invariant subsets} and \textit{invariant minimizing measures} under the Lagrangian dynamics. Finally, we obtain the existence of \textit{calibrated curves} defined on $\RR$.

Infinite-dimensional systems (\textit{infinite systems} for short) arise in the study of mechanical systems with a large number of identical particles (e.g. fluids and gases). In these models, the number of particles is infinite and the state
of the system is determined by a probability measure or, alternatively, by a random variable. The evolution of the system is characterized by an ODE in the space of probability measures or the space of random variables. A well-known example of such an ODE is the non-linear Vlasov system \cite{braunhepp'77,dobrushin'79,maslov'78,tudgangbo'10a,tudgangbo'10b,tudgangbo'12}.

In the seminal papers \cite{tudgangbo'10a,tudgangbo'10b}, Gangbo and Tudorascu introduced and developed the weak KAM theory for infinite systems. Next, they considered an infinite system of particles on the torus $\Tt^1$ and modeled it using $L^2([0,1])$ functions as random variables. They introduced the \textit{infinite-dimensional torus} and proved a \textit{Weak KAM theorem} on it. Subsequently, in \cite{tudgangbo'12}, these authors addressed the higher-dimensional case on the torus $\Tt^d$ for $d> 1$. This case is studied using probability measures over $\Tt^d$. They generalize core aspects of the one-dimensional problem to higher dimensions.

In \cite{bessi'12,bessi'13,bessi'14,bessi'14a}, Bessi examined 
infinite systems in the framework of the Aubry-Mather theory. In particular, in \cite{bessi'13}, the author studied the Aubry-Mather minimal measure theory in the infinite-dimensional setting for $d=1$.

Crucial aspects of the previous results are the following. Firstly, in the one-dimensional case, the existence of optimal trajectories is proved for monotone and square integrable initial configurations of particles. This technique is used to overcome the fact that $L^2([0,1])$ is not locally compact. Monotonicity yields compactness, and that makes it possible to extend finite-dimensional methods. Unfortunately, this technique does not generalize to higher dimensions.

Secondly, in prior publications, the higher-dimensional case was studied via the probability measures approach. The space of probability measures is a metric space, and it does not have a natural linear structure. This fact creates additional difficulties. 
A standard solution to define the velocity of a curve
is to consider velocity fields of minimal norm \cite{ambgigsavare'08}. As pointed out in \cite{tudgangbo'12}, this is not the suitable notion to develop the Weak KAM on the space of probability measures. The appropriate derivatives are the $c$-\textit{minimal velocity fields}. However, these  depend on the choice of $c \in \RR ^d$. Therefore, the definition of the viscosity solution of the cell problem depends on  $c$.

Thirdly, for any dimension $d\geq 1$, there exist \textit{weakly invariant minimizing measures} (or \textit{minimizing holonomic measures}) \cite{tudgangbo'10a,tudgangbo'12}.
Moreover, if $d=1$ there exist \textit{invariant} (or \textit{strongly invariant}) \textit{minimizing measures}  \cite{bessi'13}. The existence of invariant minimizing measures in the case $d > 1$ was not settled previously. 

Finally, the Lagrangians considered in previous publications are  mechanical Lagrangians that are the sum of kinetic and potential energy.  

In this paper, we contribute to the existing results in several directions. For any dimension $d\geq 1$, we address the following points:
\begin{itemize}
	\item[i.] For any $c \in \RR ^d$ and for generic initial configurations of the particles, we prove the existence of the optimal trajectories for the discounted cost infinite horizon problem.
	\item[ii.] For any $c\in \RR^d$, we prove that the infinite-dimensional cell problem admits a viscosity solution $U$ and that this solution is a fixed point for Lax-Oleinik semigroup. Moreover, we show the existence of $(U,c,L)$-calibrated curves defined on $\RR ^+$, where $L$ is the Lagrangian of the system.
	\item[iii.] We show the existence of invariant minimizing measures and $(U,c,L)$-calibrated curves defined on $\RR$.
	\item[iv.] We consider general Lagrangians (i.e., infinite-dimensional Tonelli Lagrangians).
\end{itemize}

In what follows, we present the statements of our main results and give a detailed description of our methods.

\subsection{Main results and the outline of the paper}

Let $I=[0,1]^d$ and $\lambda _0$ be the Lebesgue measure restricted to $I$.
A curve $\sigma : (0, T) \rightarrow {L^2(I; \RR^d)}$,  $t \mapsto \sigma_t$,
is absolutely continuous if there exists $\beta \in L^1(0,T)$ such that
\begin{equation}\label{acineq}
\|\sigma_t -\sigma_s\| \leq \int_s^t{\beta(u) du},
\end{equation}
for every $s < t$ in $(0,T)$. We denote by $AC^2(0,T ; L^2(I; \RR^d))$ the set of all absolutely continuous paths $\sigma : (0, T) \rightarrow L^2(I; \RR^d)$ such that there exists $\beta \in L^2(0,T)$ satisfying \eqref{acineq}. The set $AC^2_{loc}(0,\infty ; L^2(I; \RR^d))$
is the set of curves $\sigma:(0, \infty) \rightarrow {L^2(I; \RR^d)}$, whose restriction
to $(0,T)$ is in $AC^2(0,T ; L^2(I; \RR^d))$, for all $T>0$.

In this paper, we assume that the Lagrangian $L : \ltord \times \ltord \to \RR$ satisfies the conditions $i)-viii)$ given in Subsection \ref{assumptions}.
An important example that satisfies these conditions is 
the mechanical Lagrangian
\begin{equation*}L(M,N)=\frac{\|N\|^2_{L^2(I; \RR^d)}}{2}-\int_{I \times I}W(M(z) -M (\bar{z}))d\lambda_0(z) d\lambda_0(\bar{z}),
\end{equation*}
where $W \in C^2(\Tt ^d)$ is an interaction potential.
	



In Section \ref{sec: prelimHJeq}, we consider the discounted-cost infinite-horizon problem.
For $c \in \RR ^d$, we set
\begin{equation}\label{eqLC}
L_c(M,N)=L(M,N) + \int_I c \cdot N d\lambda_0.
\end{equation}
We fix $\varepsilon>0$, and  
for a trajectory $\tx \in AC^2_{loc}((0,\infty) ; L^2(I; \RR ^d))$,
we define the action
\begin{equation*}
\mathcal{A}_{\varepsilon}(\tx) := \int_{0}^{\infty}{ e^{- \varepsilon t} L_c(\tx, \dot{\tx}) dt}.
\end{equation*}
Since $L_c$ is bounded by below, $\mathcal{A}_{\eps}$ is well defined. Set
\begin{equation}\label{Eq: InfiniteHorizone}
V_{\varepsilon}(M) := \inf_{\tx} \{ \mathcal{A}_{\varepsilon} (\tx) : \tx \in AC^2_{loc}((0, \infty); L^2(I; \RR ^d)),\ \tx(0)=M \}.
\end{equation}

The Hamiltonian is the Legendre transform of $L$ given by
\begin{equation}
	\label{eq: HamiltonianDef}
H(M,P)=\sup_{N \in \ltord}\{-\langle P,N \rangle-L(M,N)\},
\end{equation}
for $(M,P) \in \ltord \times \ltord$. The Hamiltonian $H_c$ associated 
with $L_c$ is given by
\begin{equation*}\label{eq: Hc}
H_c(M,P)=H(M,P+c).
\end{equation*}

We have that (see \cite{crandalllions'86a, ishii'92, tataru'92, gomesnurbekyan'12a,bessi'13}) $V_{\varepsilon}$ is a viscosity solution of
\begin{equation}\label{Eq: HJinfinitehorizon}
\eps V_{\eps}(M)+H_c(M, \nabla V_{\eps}(M))= 0.
\end{equation}

Our first result is:
\begin{theorem}
	\label{TEO1}
For every differentiability point $M$ of $V_{\eps}$, there exists a unique minimizer $\tx ^* \in C^1\left([0,\infty), \ltord \right)$ of \eqref{Eq: InfiniteHorizone} with the initial condition $\tx^*(0)=M$. Furthermore, $\dot{\tx}^*(0)=-D_p H_c(M,\nabla V_\eps (M))$ and $\tx^*$ solves the Euler-Lagrange equation, that is
\begin{equation*}\label{eq: ELforVeps}
 \frac{d}{ds} \left(e^{-\eps s}D_v L_c(\tx^*, \dot{\tx}^*)\right)= e^{-\eps s} D_xL_c(\tx^*, \dot{\tx}^*).
\end{equation*}
\end{theorem}
To date, 
the existence of the minimizers for higher dimensions has been an open problem (see Remark 3.4 in \cite{tudgangbo'10a}). Our proof is based on the techniques that we developed regarding the existence of minimizers of the optimal control problem in Hilbert spaces \cite{nurbekyan'12, gomesnurbekyan'12a}. 

In Section \ref{sec: thecellproblem}, we present the proofs of our main results. Firstly, we extract a convergent subsequence out of the family of the functions $\{\eps V_{\eps}\}$ and $\{U_{\eps}:=V_{\eps}-\inf V_{\eps}\}$
\begin{equation}
\label{udef}
-\bar{H}(c)=\lim\limits_{\eps \to 0} \eps V_{\eps}, \qquad U=\lim_{\eps\to 0} U_{\eps}.
\end{equation}
We prove the following theorem.
\begin{theorem}
\label{Thm: MainU} 
Let $U$ and $\bar{H}(c)$ be given by \eqref{udef}. 
Then 
$U$ is a viscosity solution of the equation
\begin{equation*}
\label{cellP}
H_c(M,\nabla U)=\bar{H}(c).
\end{equation*}
Additionally, for every differentiability point $M \in L^2(I; \RR ^d)$ of $U$, there exists a unique trajectory $\tx^* \in C^1([0,\infty);L^2(I; \RR ^d))$
whose restriction to the interval $[0,T]$ 
 is a minimizer of
\begin{equation}\label{Eq: MinProblemforU} U(M)=\inf \{\int_0^T{L_c(\tx(s), \dot{\tx}(s))+\bar{H}(c) ds}+U(\tx(t)); \tx(0)=M \},
\end{equation}
for any $T>0$. The infimum is taken over the curves in $AC^2((0,T);L^2(I; \RR ^d))$.
Moreover, $\tx^*$ satisfies the Euler-Lagrange equation
\begin{equation*}
\frac{d}{ds} D_v L_c(\tx^*, \dot{\tx}^*)=D_x L_c(\tx^*, \dot{\tx}^*),
\end{equation*}
with $\tx^*(0)=M$ and $\dot{\tx}^*(0)= -D_p H_c(M, \nabla U(M))$.

\end{theorem}

Our next key result is:     
\begin{theorem}
\label{t1p3}
For any $c \in \RR ^d$, there exists a closed infinite-dimensional subset $\Omega$ of the tangent bundle $\mathcal{T}\ltord$ that is invariant under the Euler-Lagrange flow \eqref{eq: InfDynamics}.   
\end{theorem}
As corollaries to this Theorem,  we get:
\begin{corollary}\label{cortwosidedmin}
	For any $c\in \RR ^d$, there exist $(U,c,L)$-calibrated curves $\tx \in AC^2(\RR, \ltord)$,  that is, \eqref{eq: twosidedmins} holds.
\end{corollary}
\begin{corollary}\label{minmeasures}
	For any $c \in \RR^d$, we have
	\begin{equation*}\label{inf-H(c)}
	-\bar{H}(c)=\inf \limits_{\mu}\int\limits_{\Omega}L_cd\mu,
	\end{equation*}
	where the infimum is taken over all invariant probability measures $\mu$ on $(\Omega, \mathfrak{B})$. Moreover, the infimum is achieved.
\end{corollary}
\begin{remark}
	$\mathfrak{B}$ is the $\sigma$-algebra of subsets of $\Omega$ that are closed under measure-preserving transformations and integer translations. See \eqref{Bsigmalgebra} for the definition.
\end{remark}

\noindent {\bf Thanks:} We thank Wilfrid Gangbo for his valuable comments on this manuscript. 

\section{Preliminaries and main assumptions}

Here, we present background material on mechanical systems with finite
 or
infinite number of identical particles. 

\subsection{Mechanical systems with a finite number of indistinguishable particles}

Consider a system of $n$ identical particles on the torus $\Tt ^d$. Let $l: (\RR ^d)^n \times (\RR ^d)^n \to \RR$ be the corresponding Lagrangian. Denote by $x_i(t) \in \RR$ the position of the particle $"i"$ at time $t$. Let $(x_i^0,v_i^0)$ be the corresponding initial position and velocity. From \emph{Hamilton's minimal action principle}, the system evolves according to the Euler-Lagrange equation:  
\begin{equation}\label{eq: FinDynamics}
\begin{cases}
\frac{d}{dt} D_v l(x(t),\dot{x}(t))= D_x l(x(t),\dot{x}(t)),\\
x_i(0)=x_i^0,\ \dot{x}_i(0)=v_i^0,\ i=1,2,\ldots,n,
\end{cases}
\end{equation}
where $x(t)=(x_1(t),x_2(t), \cdots, x_n(t)) \in (\RR ^d)^n$. Let $h:(\RR ^d)^n\times(\RR ^d)^n \to \RR$ be the Hamiltonian given by the Legendre transform
\begin{equation*}
  h(x,p)=\sup \limits _{v \in (\RR ^d)^n} \{-p \cdot v- l(x,v)\}.
\end{equation*}
Then, \eqref{eq: FinDynamics} has the equivalent formulation in Hamiltonian form
\begin{equation}\label{eq: FinHmltForm}
\begin{cases}
\dot{x}(t) = -D_p h(x(t),p(t)),\\
\dot{p}(t) = D_x h(x(t),p(t)),\\
x_i(0)=x_i^0, \xd_i(0)=p_i^0,\quad i=1,2,\ldots,n,
\end{cases}
\end{equation}
where $p(t)=(p_1(t),p_2(t),\cdots p_n(t)) \in (\RR ^d)^n$ is referred to as the momentum.

A critical issue in classical mechanics is the study of qualitative properties of  \eqref{eq: FinDynamics} and \eqref{eq: FinHmltForm}. Since the particles move on the torus $\Tt ^d$, we assume that $l$ (and, consequently, $h$) are periodic in the position variable, $x$. Because the particles are identical, $l$ and $h$ are invariant under permutations, that is, 
for all points $(x_i,v_i) \in \RR ^d \times \RR ^d, \ i=1,2,\cdots,n$ and all permutations $\sigma \in S_n$, 
\begin{align*}
	& l(x_1,x_2,\cdots , x_n, v_1, v_2, \cdots v_n)=l(x_{\sigma(1)},x_{\sigma(2)},\cdots , x_{\sigma(n)}, v_{\sigma(1)}, v_{\sigma(2)}, \cdots v_{\sigma(n)})\\
	& h(x_1,x_2,\cdots , x_n, v_1, v_2, \cdots v_n)=h(x_{\sigma(1)},x_{\sigma(2)},\cdots , x_{\sigma(n)}, v_{\sigma(1)}, v_{\sigma(2)}, \cdots v_{\sigma(n)}).
\end{align*}
Then, \eqref{eq: FinDynamics} and \eqref{eq: FinHmltForm} can be viewed as dynamical systems on $(\Tt ^d)^n / S_n$.

Graphs of closed one-forms that lie in the level sets of the Hamiltonian are invariant under the flow \eqref{eq: FinHmltForm}. Since closed forms on $(\Tt ^d)^n / S_n$ are given by $\omega _x(p)=  \sum \limits _{i=1}^{n}c \cdot p_i + \langle D u(x), p \rangle$, for some $c \in \RR ^d$ and $u : (\RR ^d )^n \to \RR$ periodic, we are led to the equation
\begin{equation}\label{eq: cellproblem}
 h(x, D u(x)+\tilde{c})= \lambda,
\end{equation}
where $\lambda \in \RR$ is a constant, and $\tilde{c}=(c,c,\cdots,c) \in (\RR ^d)^n$. Equation \eqref{eq: cellproblem} is the cell problem associated with \eqref{eq: FinHmltForm}.

Qualitative properties of \eqref{eq: FinDynamics} and \eqref{eq: FinHmltForm} are closely linked to the regularity properties of the solutions to \eqref{eq: cellproblem}.

\subsection{The random variable approach}

A standard method for studying mechanical systems with an infinite
number of identical particles is to look at a probability measure encoding the positions of the particles. The evolution of a system is a curve in a space of probability measures. For problems with a finite
number of particles, this measure is the empirical measure of the particles' positions.
If the ambient space is compact, the space of probability measures on it is also compact. Compactness is particularly relevant for a variational theory such as the Weak KAM theory. On the other hand, the lack of a linear structure makes it more complex to introduce notions such as the derivative of a path.    

An alternative approach consists of
regarding the state of the system as a random variable. Each realization of this random variable
represents the position 
of one particle. 
The evolution of the system is given by a trajectory in a space of random variables.
The space of random variables is a vector space and has a natural Riemannian structure. Unfortunately, in contrast to the space of probability measures, 
non-trivial spaces of random variables are not locally compact.
However, for symmetrical problems, random variables that have the same law represent an equivalent state of the system. Thus, the dynamics can be viewed as an evolution in the quotient space of random variables with respect to the equivalence relation of having the same law. This latter space is compact and isometric to the space of probability measures \cite{tudgangbo'12}. Hence, we can use compactness arguments.

In \cite{tudgangbo'10a, tudgangbo'10b} the random variables approach is used by working in $L^2([0,1], \RR)$. Let $\mathcal{P}_2(\RR)$ be the space of probability measures over $\RR$ with finite second-order moment endowed with the 2-Wasserstein distance $W_2$. Then, $\mathcal{P}_2(\RR)$ is isometric to the set of monotone non-decreasing functions in $L^2([0,1],\RR)$. The lack of compactness of $L^2([0,1], \RR)$ is offset by the compactness of the set of monotone functions via Helly's selection theorem.


If $d >1$, the random variables approach leads to a dynamical system on $\ltord$, where $I=[0,1]^d$. Unfortunately, unlike in dimension $1$, there is no canonical isometry between $\mathcal{P}_2(\RR ^d)$ and some subset of $\ltord$. Hence, the methods used in \cite{tudgangbo'10a, tudgangbo'10b} cannot be applied if $d > 1$. In particular, the existence of minimizing curves for both the discounted infinite horizon problem and the Lax-Oleinik semigroup were open until now. In this paper, we prove the existence of minimizers in the general case $d \geq 1$. In \cite{tudgangbo'12}, the authors use an alternative approach and work directly in $\mathcal{P}(\Tt ^d)$. 
%
Our techniques are more functional analytic in spirit, and use results from the calculus of variations in Hilbert spaces \cite{gomesnurbekyan'12a}.

\subsection{Mechanical systems with an infinite number of indistinguishable particles via random variables}

Consider a mechanical system with an infinite number of identical particles. Assume that there is a one-to-one correspondence between  particles and points in $I=[0,1]^d$. 
We encode the positions of the particles in a 
random variable $M \in \ltord$. Using the notation of \cite{tudgangbo'10a, tudgangbo'10b},
for each point $z \in I$, $Mz \in \RR ^d$ is the position of the particle $"z"$ in the space.

Let $L: \ltord \times \ltord \to \RR$ be the Lagrangian of the system. The
associated dynamics is given by the Euler-Lagrange equation
\begin{equation}\label{eq: InfDynamics}
\begin{cases}
\frac{d}{dt} D_v L (\tx(t), \dot{\tx}(t)) = D_x L(\tx(t), \dot{\tx}(t)),\\
\tx(0)=M, \dot{\tx}(0)=N,
\end{cases}
\end{equation}
where the partial derivatives are in Fr\'{e}chet sense, and $(M,N) \in \ltord \times \ltord$ is the initial configuration of positions and velocities.
%
For $H$ as in \eqref{eq: HamiltonianDef}, 
the infinite-dimensional Hamiltonian system is then
\begin{equation}\label{eq: InfHmltForm}
\begin{cases}
\dot{\tx}(t) = -D_p H(\tx(t),\tp(t)),\\
\dot{\tp}(t) = D_x H(\tx(t),\tp(t),\\
\tx(0)=M, \ \ \tp(0)=P.
\end{cases}
\end{equation}

As in the finite-dimensional case, we need the notions of ``periodicity" and ``invariance under permutations" for the Lagrangian and Hamiltonian. Consider the subset of $L^2(I; \RR ^d)$
$$\ltordz:=\{M \in L^2(I; \RR ^d) \  ;\  Mz \in \ZZ ^d,\, \  \lambda _0 \  a.e. \}.$$
This set is a subgroup with respect to addition. A function $F$ defined on $L^2(I; \RR ^d)$ is called periodic if $F(M+Z)=F(M)$ for all $M \in L^2(I; \RR ^d)$ and $Z \in \ltordz$. Periodicity of the Lagrangian $L$ in the spatial variable means that
$$L(M+Z,N)=L(M,N),$$
for all $M,N \in \ltord$ and $Z \in \ltordz$. The \emph{$d$-infinite-dimensional torus} is the quotient space
\begin{equation*}
  \TT ^d : = \ltord / \ltordz.
\end{equation*}
	Let $(X,\mathfrak{F})$ and $(Y,\mathfrak{G})$ be measurable spaces, and $M : (X, \mathfrak{F}) \to (Y,\mathfrak{G})$ a measurable map. 
	Suppose $\mu$ is a measure on $(X,\mathfrak{F})$.
	The push-forward of the measure $\mu$ through the map $M$ is the measure
	 $\nu=M \sharp \mu$ on $Y$ given by $\nu [C] = \mu[M^{-1}(C)]$ for all sets $C \in \mathfrak{G}$. 

Consider the set $\mathcal{G}$ of all bijective functions $G:I \to I$ such that $G$ and $G^{-1}$ are Borel measurable and that push-forward the Lebesgue measure $\lambda _0$ to itself. Then $\mathcal{G}$, equipped with the composition operation, is a non-commutative group that plays the role of $S_n$ in the infinite-dimensional setting. Hence, invariance under permutations of the Lagrangian $L$ in the infinite-dimensional setting is the invariance under the action of $\mathcal{G}$:
\begin{equation*}L(M \circ G, N \circ G)=L(M,N),
\end{equation*}
for all $M,N \in L^2(I; \RR ^d)$ and $G \in \mathcal{G}$. We call this property \emph{rearrangement invariance}. If $L$ is periodic and rearrangement invariant, the Euler-Lagrange equation \eqref{eq: InfDynamics} is a dynamical system on the \emph{$d$-infinite-dimensional symmetrical torus $\TT ^d / \GG$}.

A thorough analysis of the symmetrical torus $\TT^d / \GG$ can be found in \cite{tudgangbo'10a} ($d=1$) and in \cite{tudgangbo'12} ($d>1$). Here, we recall several important facts that we require for our analysis.

We endow $\TT^d / \GG$ with the induced metric $\text{dist}_{weak}$ defined as
\begin{equation*}\label{def: distweakont^d/g}
\text{dist}_{weak}(M_1,M_2) = \inf_{G \in \mathcal{G}, Z \in \ltordz}{\|M_1- M_2 \circ G-Z\|}.
\end{equation*}
This distance satisfies all the axioms of a metric distance except the non-degeneracy, that is,
 there exist $M_1,M_2 \in \ltord$ such that $\text{dist}_{weak}(M_1,M_2)=0$ but $M_1\neq M_2$.

We define an equivalence relation as follows: $M_1,M_2 \in \ltord$ are equivalent, denoted by $M_1 \sim M_2$, if
\begin{equation*}\label{eqrel}
	\dw(M_1,M_2)=0.
\end{equation*}
It is straightforward to see that $\sim$ is an equivalence relation. Define $\SS^d$ as
\begin{equation*}
	\SS^d=\left(\TT^d /\GG\right) /\sim. 
\end{equation*}
$\SS^d$ is a metric space with the induced distance
\begin{equation*}
	\text{dist}_{\SS}(M_1,M_2)=\dw(M_1,M_2).
\end{equation*}
\begin{remark}
	By the abuse of notation, we denote by $M$ all equivalence classes of $M$.
\end{remark}

\begin{proposition}[\cite{tudgangbo'10a,tudgangbo'12}]\label{Prp: GInvariant}
	The space $(\SS ^d,\text{dist}_{\SS})$ is isometric to  $(\Pp(\Tt^d),W_2)$, where $W_2$ is the 2-Wasserstein distance. Consequently, $\SS ^d$ is a compact, complete, separable metric space.
	
	Furthermore, for any continuous periodic function, $F: \ltord \rightarrow \RR$, the following assertions are equivalent:
	\begin{enumerate}
		\item[i)]
		$F$ is rearrangement invariant 
		\item[ii)]
		$F(M_1)=F(M_2)$ for all $M_1,M_2 \in L^2(I; \RR ^d)$ such that $M_1 \sim M_2$.
	\end{enumerate}
\end{proposition}
Finally, we set
\begin{equation}\label{proj}
\pi:\ltord \to \SS^d
\end{equation}
to be the natural projection that maps a function $M$ to its equivalence class. Note that $\pi$ is 1-Lipschitz.

\subsection{Main assumptions}\label{assumptions}

Here, we suppose that $L : \ltord \times \ltord \to \RR$ satisfies the following conditions, for some constants $C,K_L,\gamma >0$ and for all $M,N,H_1,H_2 \in \ltord, \ G \in \mathcal{G}, \ Z \in \ltordz$, 
\begin{itemize}
	\item[i)] $L(M+Z,N)=L(M,N)$ (periodicity);
	\item[ii)] $L(M \circ G, N \circ G) = L(M,N)$ (rearrangement invariance);
	\item[iii)] $L \geq 0$;
	\item[iv)] $L$ is $C^1$ Fr\'{e}chet differentiable and $L,DL$ are locally uniformly continuous, where $DL$ is the full derivative of $L$ in Fr\'{e}chet sense;
	\item[v)] $L(M,N) \leq C(1+\|M\|^2+\|N\|^2), \ |L(M,0)| \leq C$;
	\item[vi)] $\|DL\|\leq C+C L$;
	\item[vii)] $L(M+H_1,N+H_2)-L(M,N)-\langle D_{x}L(M,N), H_1\rangle-\langle D_{v}L(M,N), H_2\rangle \geq \gamma \|H_2\|^2-K_L \|H_1\|^2$;
	\item[viii)] $L(M+H_1,N+H_2)-L(M,N)-\langle D_{x}L(M,N), H_1\rangle-\langle D_{v}L(M,N), H_2\rangle \leq K_L \|H_2\|^2+K_L \|H_1\|^2$.
\end{itemize}
For $c \in \RR ^d$, let $L_c$ be as in \eqref{eqLC}.
The Hamiltonian, $H$, associated with the Lagrangian, $L$, is given by \eqref{eq: HamiltonianDef}.
We refer to the second variable of the Hamiltonian, $P$, as the \textit{momentum} variable. Differentiation with respect to the momentum variable is denoted by $D_p$. Differentiation with respect to the first variable $M$ is denoted by $D_x$.

Assumptions i)-viii) yield that $H$ is $C^1$ in Fr\'{e}chet sense, strictly convex, and coercive. Furthermore,  
\begin{equation}\label{eq: Legendretransform:H-L}
L(M,N) = \sup \{ -\langle P, N \rangle - H(M,P) \ ; \ P \in \ltord \}.
\end{equation}
Let $N$ and $P$ be the maximizers in \eqref{eq: HamiltonianDef} and \eqref{eq: Legendretransform:H-L}. Then, they are related by the Legendre transform (which is one-to-one from $\ltord$ to itself)
\begin{equation}\label{eq: legtrans}
P=-D_v L(M,N), \quad N=-D_p H(M,P).
\end{equation}
For $N$ and $P$ satisfying \eqref{eq: legtrans}, we have
\begin{equation*}
D_x H(M,P)=-D_x L(M,N). 
\end{equation*}

These duality statements can be found in \cite{FS}, in the finite-dimensional case. Similar techniques apply to the infinite-dimensional case.

\section{The discounted-cost infinite-horizon problem}\label{sec: prelimHJeq}

In this section, we study the discounted-cost infinite-horizon problem. 
As is standard in Weak KAM theory, this problem can be used to build solutions 
to the cell problem \cite{lionspapvarad'88}. 

Recall that a function taking values on $\RR\cup \{\pm \infty\}$ is {\em proper} if it is not identically $\pm\infty$. 

\begin{definition} Let $V:L^2(I; \RR ^d)\to \mathbb{R} \cup \{ \pm \infty \}$ be a proper function and $M \in L^2(I; \RR ^d)$ be a point in its domain. Then, a vector $\xi \in \ltord$ 
\begin{itemize}
 \item[i)] is a subdifferential of $V$ at $M$ if $V(M+X)\geq V(M)+ \langle \xi, X \rangle + o(||X||)$;
 \item[ii)] is a superdifferential of $V$ at $M$, if $V(M+X) \leq V(M)+\langle \xi, X \rangle + o(||X||)$. 
\end{itemize}
The set $D^{-}V(M)$ (resp. $D^{+}V(M)$) is the set of subdifferentials (resp.  superdifferentials) at $M$.
\end{definition}
\begin{remark}If the sets $D^{-}V(M)$ and $D^{+}V(M)$ are simultaneously non-empty, then $V$ is differentiable at $M$ and $D^{-}V(M)=D^{+}V(M)=\{\nabla V(M)\}$.
\end{remark}
Let $F:\ltord \times \RR \times \ltord \to \RR$ be a continuous function. 
Consider the first-order infinite-dimensional partial differential equation 
\begin{equation}\label{Eq: HJ}
F(M, V, \nabla V(M))=0.
\end{equation}
\begin{definition} A continuous function $V : L^2(I; \RR ^d) \rightarrow \mathbb{R}$ is a
\begin{itemize}
 \item[i)] viscosity subsolution for \eqref{Eq: HJ} if $F(M,V(M), \zeta) \leq 0$, for all $M \in L^2(I; \RR ^d)$ and all $\zeta \in D^{+}V(M)$;
 \item[ii)] is a viscosity supersolution for \eqref{Eq: HJ} if $F(M,V(M), \zeta) \geq 0$, for all $M \in L^2(I; \RR ^d)$ and all $\zeta \in D^{-}V(M)$;
 \item[iii)] is a viscosity solution for \eqref{Eq: HJ} if $V$ is both a subsolution and a supersolution for \eqref{Eq: HJ}.
\end{itemize}
\end{definition}

Let $V_{\varepsilon}$ be the discounted value function given by \eqref{Eq: InfiniteHorizone}.
Since the Lagrangian $L_c$ is rearrangement invariant and periodic in the spatial variable, so is 
the value function $V_{\varepsilon}$.

We collect several elementary properties of the value function $V_{\eps}$ in the following proposition.
\begin{proposition}\label{vepsproperties}
	For any $t>0$ and $M \in L^2(I; \RR ^d)$, we have that
	\begin{equation}\label{Eq: DPP}
	V_{\eps}(M)= \inf\{ {\int_0^t{e^{-\eps s} L_c(x, \dot{x}) ds} + e^{-\eps t} V_{\eps}(x(T))} ; x(0)=M \}.
	\end{equation}
	Furthermore, $V_{\eps}$ is a viscosity solution of the Hamilton-Jacobi equation \eqref{Eq: HJinfinitehorizon}.
	
	Moreover, 
	\begin{itemize}
		\item [i)] The family of functions $\{\eps V_{\eps}\}$ is uniformly bounded.
		\item [ii)] For every $\eps >0$ the function $V_{\eps}$ is Lipschitz continuous with Lipschitz constant independent of $\eps$.
		\item [iii)] For every $\eps >0$ the function $V_{\eps}$ is semiconcave.
	\end{itemize}
\end{proposition}
\begin{proof} In the finite-dimensional case, 
these facts	are standard and are discussed, for instance, in \cite{FS, Bardi, Barl}. In the infinite-dimensional setting, the same methods can be applied without changes. 
Properties of the value function are examined, in the context of viscosity solutions, in \cite{crandalllions'86a, ishii'92, tataru'92, bessi'13, nurbekyan'12, gomesnurbekyan'12a}.
\end{proof}
\begin{corollary}\label{crl: Vdiff} The superdifferential $D^{+}V_{\eps}$ is nonempty at every point $M \in L^2(I; \RR ^d)$. Besides, $V_{\eps}$ is Fr\'{e}chet differentiable on an everywhere dense $G_\delta$ set.
\end{corollary}
\begin{proof}
A convex function on a Banach space has a non-empty subdifferential at every point where it is finite and continuous \cite{minty'64}. Moreover, if the Banach space is also a \textit{strong differentiability space} \cite{asplund'68}, then every convex function defined on it is Fr\'{e}chet differentiable on a $G_{\delta}$ dense subset of its domain of continuity. $\ltord$ is a strong differentiability space (Theorem 1, \cite{asplund'68}). Furthermore, $V_{\eps}$ is semiconcave, finite and everywhere continuous. Accordingly, $D^{+}V_{\eps}(M) \neq \emptyset$ for all $M \in L^2(I; \RR ^d)$ and $V_{\eps}$ is Fr\'{e}chet differentiable on a $G_\delta$ dense subset of $\ltord$.
\end{proof}

In \cite{tudgangbo'10a}, the authors proved that, in the one-dimensional case, when $M$ is monotone non-decreasing, \eqref{Eq: InfiniteHorizone} admits a minimizer in $H^2_{loc}((0,\infty);L^2(I))$ that satisfies the Euler-Lagrange equation. Here, we establish
the existence of minimizers on an everywhere dense $G_\delta$ subspace of $L^2(I; \RR ^d)$, for any $d\geq 1$. 

Next, we detail the proof of 
the main result of this section, Theorem \ref{TEO1}. 
\begin{proof}[Proof of Theorem \ref{TEO1}]
In \cite{gomesnurbekyan'12a}, we studied the finite horizon optimal control problems in Hilbert spaces. We proved that at every point of differentiability of the value function, there exists a unique $C^1$ minimizer (Theorem 6.2, \cite{gomesnurbekyan'12a}). Since the infinite horizon problem can be seen as a finite horizon one, the existence of $\tx^*$ is a direct consequence of that result.
It is also standard that minimizers solve the Euler-Lagrange equation \eqref{eq: ELforVeps} \cite{nurbekyan'12}.
\end{proof}

\section{The infinite-dimensional weak KAM theory}\label{sec: thecellproblem}

In this section, we prove our main results: Theorem \ref{Thm: MainU}, Theorem \ref{t1p3}, Corollary 1.1, and Corollary 1.2.

Closed one-forms on $\TT ^d / \GG$ are given by $DU + c \chi _I$, for some periodic function $U : \ltord \to \RR$ and some $c \in \RR ^d$ \cite{tudgangbo'10a,tudgangbo'12}. Hence, the cell problem associated with \eqref{eq: InfDynamics} is 
\begin{equation}\label{eq: CellProbIntro}
H(M,DU+c \chi _I)=\lambda,
\end{equation}
where $\lambda \in \RR$. Moreover,
as stated in Theorem \ref{Thm: MainU}, for every $c \in \RR ^d$ there exists a unique number $\lambda= \bar{H}(c)$ such that \eqref{eq: CellProbIntro} has a periodic rearrangement invariant viscosity solution $U$.
In Proposition \ref{Prp: OptCntrlforU}, we prove that
this solution is a fixed point of the \emph{Lax-Oleinik semigroup}, that is,
\begin{equation}\label{eq: CellProbVarIntro}
U(M)=\inf_{\tx\in AC^2((t,t_1), L^2(I; \RR ^d))} \left\{\int_t^{t_1}{\left( L_c(\tx(s), \dot{\tx}(s))+\bar{H}(c) \right) \ ds}+U(\tx(t_1));\ \tx(t)=M\right\}
\end{equation}
for any $M \in \ltord$ and $t<t_1$.
The case $d=1$ was studied in \cite{tudgangbo'10a, tudgangbo'10b}, in a slightly weaker form in what concerns the Lax-Oleinik semigroup. Analogous results are available on the space of probability measures in \cite{tudgangbo'12}.

Additionally, we show that $U$ is semiconcave, and hence Fr\'{e}chet differentiable on a $G_ {\delta}$ everywhere dense set (Proposition \ref{Prp: USemiconcavity}). Furthermore, at differentiability points $M$ of $U$, the infimum in \eqref{eq: CellProbVarIntro} is attained at a $C^1$ minimizer (Theorem \ref{Thm: MainU}). This issue was settled for $d=1$ in  \cite{tudgangbo'10a} using different ideas, and the higher-dimensional case was not addressed there. A corresponding result on the space of probability measures can be found in \cite{tudgangbo'12}.

A curve $\tx : [t_0,t_1] \to \ltord$ is called a \emph{$(U,c,L)$-calibrated curve} if
\begin{equation*}
U(\tx(\beta))-U(\tx(\alpha)) = \int_{\beta}^{\alpha}{\left( L_c(\tx(s), \dot{\tx}(s))+\bar{H}(c) \right) \ ds}
\end{equation*}
for all $\alpha, \beta \in [t_0,t_1]$. Here,
we prove that for any differentiability point $M$ of $U$ there exists a calibrated curve defined on $[0, \infty)$ starting at $M$.

\subsection{The cell problem: existence of solutions and elementary properties}

We begin by considering the limit as $\eps \to 0$ of the solutions $V_\eps$ 
to \eqref{Eq: HJinfinitehorizon}.
\begin{proposition}\label{Prp: UandH(c)}
	Let $\eps>0$ and
	$V_\epsilon$ be 
	a solution to \eqref{Eq: HJinfinitehorizon}.
	Define $U_{\eps}:=V_{\eps}-\inf V_{\eps}$.
Then
\begin{itemize}
 \item[i)] the function $U_{\eps}$ is rearrangement invariant for every $\eps>0$. Furthermore, the family of functions $\{ U_{\eps} \}$ is uniformly Lipschitz continuous.
 \item[ii)] The family of functions $\{U_{\eps}\}$ has a uniformly convergent subsequence with a Lipschitz continuous limit $U$. Additionally, the family of functions $\{\eps V_{\eps}\}$ has a uniformly convergent subsequence with constant limit depending on $c$: $-\bar{H}(c)$.
\end{itemize}
\end{proposition}

\begin{remark}
A priori, the constant limit of the convergent subsequence of the $\{\eps V_{\eps}\}$ is not unique, and Proposition \ref{Prp: OptCntrlforU} is valid for any such limit and corresponding limit function $U$. However, it is simple to check that\eqref{Eq: MinProblemforU} implies the uniqueness of such a constant. Hence, $\bar{H}(c)$ is uniquely determined by the vector $c \in \RR^d$.  
\end{remark}

\begin{proof} The family of functions $V_{\eps}$ is equilipschitz (Proposition \ref{vepsproperties}), thus, the family $U_{\eps}$ is also equilipschitz.

The Lagrangian $L_c$ is  rearrangement invariant hence $V_{\eps}$ and $U_{\eps}$ are also rearrangement invariant functions. By Proposition \ref{Prp: GInvariant}, we may identify $U_{\eps}$ and $V_{\eps}$ with functions on $\mathbb{S}^d$. Since $\mathbb{S}^d$ is a compact metric space, $U_{\eps}$ reach their minima that are $0$. Furthermore, since they are uniformly Lipschitz, we obtain that $\{U_{\eps}\}$ is bounded equicontinuous family of functions on the compact space $\mathbb{S}^d$. Therefore, by the Arzela-Ascoli Theorem, we conclude that it has a uniformly convergent subsequence. The limit $U$ is also Lipschitz continuous.

From Proposition \ref{vepsproperties}, we have that $\{\eps V_{\eps}\}$ is a uniformly bounded and equicontinuous family of functions. Hence, by the Arzela-Ascoli theorem, we obtain that it has a uniformly convergent subsequence. The limit of this subsequence has Lipschitz constant $0$, which is a constant function.
\end{proof}
\begin{proposition}\label{Prp: OptCntrlforU} For any $t>0$ and any $M \in L^2(I; \RR ^d)$,
	$U$ solves \eqref{Eq: MinProblemforU}.
\end{proposition}
\begin{proof} Fix any $M \in L^2(I; \RR ^d)$. We claim that for any $\tx \in AC^2((0,t);L^2(I; \RR ^d))$ such that $\tx(0)=M$,
\begin{equation}
\label{s1}
U(M) \leq \int_0^t{\left(L_c(\tx(s), \dot{\tx}(s))+\bar{H}(c)\right) ds}+U(\tx(t)).
\end{equation}
From \eqref{Eq: DPP}, we have that, for every $\eps>0$,
$$V_{\eps}(M) \leq \int_0^t{e^{-\eps s}L_c(\tx(s), \dot{\tx}(s))ds}+e^{-\eps t}V_{\eps}(\tx(t)).$$
Because $U_{\eps}=V_{\eps}-\inf V_{\eps}$, we have
$$U_{\eps}(M) \leq \int_0^t{e^{-\eps s}L_c(\tx(s), \dot{\tx}(s))ds}+U_{\eps}(\tx(t))+(e^{-\eps t}-1)V_{\eps}(\tx(t)).$$
Passing to the limit when $\eps \to 0$ and using Proposition \ref{Prp: UandH(c)}, we obtain 
\eqref{s1}.

Next, we prove the opposite inequality. Fix $M \in \ltord,\ t>0$.
Choose a sequence $\{ \eps_n >0\}$ converging to $0$. Let $\{\tx _n \}$ be a sequence of curves that satisfy
$$
\int_0^t{e^{-\eps_n s}L_c(\tx_n(s), \dot{\tx}_n(s))ds}+e^{-\eps_n t}V_{\eps_n}(\tx_n(t)) \leq V_{\eps_n}(M)+\frac{1}{n}.
$$
The previous inequality can be rewritten as
$$
\int_0^t{\left(L_c(\tx_n, \dot{\tx}_n)+\bar{H}(c)\right)ds}+U_{\eps_n}(\tx_n(t))+I+J \leq U_{\eps_n}(M)+\frac{1}{n},
$$
where
\begin{equation*}
\begin{cases}
I=(e^{-\eps_n t}-1)V_{\eps_n}(\tx_n(t))-t\bar{H}(c),\\
J=\int_0^t{(e^{-\eps_n s}-1)L_c(\tx_n,\dot{\tx}_n)}.
\end{cases}
\end{equation*}
Thus, if we show that $I,J \rightarrow 0$, we are done.
Due to Proposition \ref{Prp: UandH(c)}, $I \to 0$ . Assumptions v)-vii) guarantee that $L_c$ is bounded by below. Since adding a constant to $L_c$ in $J$ does not change the limit, we can assume that $L_c \geq 0$. Hence,
$$|J|=\int_0^t{(e^{\eps_n s}-1)e^{-\eps_n s}L_c(\tx_n,\dot{\tx}_n)} \leq (e^{-\eps_n t}-1) \int_0^t{e^{-\eps_n s}L_c(\tx_n,\dot{\tx}_n)} \to 0,$$
because the sequence $\int_0^t{e^{-\eps_n s}L_c(\tx_n,\dot{\tx}_n)}$ is bounded.
\end{proof}
 The function $U$ enjoys properties
 analogous to the ones satisfied by $V_\eps$, namely:
\begin{proposition}\label{Prp: USemiconcavity}$U:\ltord \rightarrow \RR$ is semiconcave. Furthermore, $U:L^2(I; \RR ^d) \rightarrow \RR$ has non-empty superdifferential $D^{+}U(M)$ at every point $M \in L^2(I; \RR ^d)$, and it is differentiable on an everywhere dense $G_{\delta}$ set.
\end{proposition}

%

We now gather the previous results and present the proof of Theorem \ref{Thm: MainU}.
\begin{proof}[Proof of Theorem \ref{Thm: MainU}]
The proofs of Proposition \ref{vepsproperties} and Theorem \ref{TEO1} apply without substantial changes due to Propositions \ref{Prp: UandH(c)}, \ref{Prp: OptCntrlforU} and \ref{Prp: USemiconcavity}. 
\end{proof}

%
%

\subsection{Existence of an invariant subset}\label{subsec: twosidedmins}

A trajectory $\tx \in AC^2_{loc}(\RR; L^2(I; \RR ^d))$ is called a two-sided minimizer (or two-sided $(U,c,L)$-calibrated curve) if
\begin{equation}\label{eq: twosidedmins}
U(\tx(t_1))=\int_{t_1}^{t_2}{\left(L_c(\tx(s),\dot{\tx}(s))+\bar{H}(c)\right) ds}+U(\tx(t_2)),
\end{equation}
for all $-\infty<t_1<t_2<\infty$.

We proceed by proving some preliminary lemmas.
\begin{lemma}\label{Lemma: DiffofPerInvfunction} Let $F:L^2(I; \RR ^d) \rightarrow \RR$ be a periodic and rearrangement invariant function. Suppose $M_1$ is a Fr\'{e}chet differentiability point of $F$ and $M_2 \sim M_1$. Then $F$ is Fr\'{e}chet differentiable at $M_2$ and $\nabla F(M_1) \sim \nabla F(M_2)$.
\end{lemma}
\begin{proof}
The proof of this lemma can be found in lectures by P.-L. Lions on
mean-field games at Coll\`{e}ge de France \cite{lions_collegedefranse, cardaliaguetMFG'12}.
\end{proof}

Let $D \subset L^2(I; \RR ^d)$ be the differentiability set of the function $U$. Denote by $\mathcal{D}:=\{(M,\nabla U(M)); M \in D \}$ the graph of the gradient of $U$.

By Theorem \ref{Thm: MainU}, for every point $M \in D$, there exists a unique trajectory $\tx$ that minimizes \eqref{Eq: MinProblemforU}. Consider the adjoint variable $\tp(s)=-D_{v}L_c(\tx(s),\dot{\tx}(s)), s \geq 0$. The trajectory $(\tx(s),\tp(s)) \subset \mathcal{T}^* L^2(I; \RR ^d)$ solves
\begin{equation}\label{Eq: HJsystemforU}
\begin{cases}
\dot{\tx}=-D_{p} H_c(\tx,\tp),\\
\dot{\tp}=D_{x} H_c(\tx,\tp)
\end{cases}
\end{equation}
with initial data $(\tx(0),\tp(0))=(M, \nabla U(M))$.

On the other hand, if $(\tx,\tp)$ satisfies \eqref{Eq: HJsystemforU} with initial data $(\tx(0),\tp(0))=(M,\nabla U(M)) \in \DD$, then $\tx$ solves the Euler-Lagrange equation $D_{\tx}L_c(\tx,\dot{\tx})=\frac{d}{dt}D_{v}L_c(\tx,\dot{\tx})$ with initial conditions $\tx(0)=M,\ \dot{\tx}(0)=-D_p H_c(M,\nabla U(M))$. Therefore, by Theorem \ref{Thm: MainU}, $\tx$ is the unique minimizer in \eqref{Eq: MinProblemforU}.

Define $\DD _t \subset \mathcal{T}^* L^2(I; \RR ^d)$ as the set of all points $(M,P) \in \mathcal{T}^* L^2(I; \RR ^d)$ for which there exists a solution $(\tx(s),\tp(s))$ of \eqref{Eq: HJsystemforU} with initial data $(\tx(0),\tp(0)) \in \DD$, and $(\tx(t),\tp(t))=(M,P)$. In other words, $\DD _t$ is the image at time $t$ of the set $\DD$ under the Hamiltonian flow \eqref{Eq: HJsystemforU}.

Since $U$ is a value function, it is differentiable along the minimizing trajectory (Corollary 4.1, \cite{gomesnurbekyan'12a}). Therefore, $\DD _t \subset \DD$, for any $t>0$. Hence, since $\DD$ is a graph, $\DD _t$ is also a graph, for all $t>0$.  Moreover, $\DD _t \subset \DD _s$ for all $s<t$.

Let $D_t$ be the projection of the set $\DD _t$ onto the spatial component of the cotangent bundle $\mathcal{T}^* L^2(I; \RR ^d)$, that is, all points $M \in L^2(I; \RR ^d)$ such that $(M,\nabla U(M)) \in \DD _t$.
\begin{lemma}\label{Lemma: InvofD_t} If $M_1 \in D_t$ for some $t>0$, then $M_2 \in D_t$, for all $M_2 \sim M_1$.
\end{lemma}
\begin{proof} Fix $M_1 \in D_t$. Thus, $M_1$ is a differentiability point of  $U$ and $(M_1,\nabla U(M_1)) \in \DD _t$. Since $U$ is periodic and invariant under measure-preserving transformations, by Lemma \ref{Lemma: DiffofPerInvfunction}, we have that $U$ is also differentiable at $M_2$. Furthermore, there exist $G_n \in \mathcal{G}$ and $Z_n \in L_{\mathbb{Z}}^2(I; \RR ^d)$ such that $M_1 \circ G_n+Z_n \to M_2$ in the strong $L^2$ sense. Denote by $M_n:=M_1 \circ G_n+Z_n$. Recall that the gradient of a convex function is continuous at all points 
	where it is defined, see \cite{asplund'68}. Then, we have that $\nabla U(M_2)=\lim \limits_{n \rightarrow \infty}{\nabla U(M_n)}=\lim \limits_{n \rightarrow \infty}{\nabla U(M_1) \circ G_n}$.

Because $(M_1,\nabla U(M_1)) \in \DD _t$, there exists a minimizing trajectory $\tx$ such that $(\tx(0), \tp(0)) \in \DD$ and $(\tx(t),\tp(t))=(M_1,\nabla U(M_1))$. Consider the trajectories $\ty_n(s)=\tx(s) \circ G_n+Z_n$ and $\tq_n(s)= \tp(s) \circ G_n$. Due to the rearrangement invariance and periodicity of the Hamiltonian $H_c$, we have that $(\ty_n,\tq_n)$ solves \eqref{Eq: HJsystemforU} with terminal data $(\ty_n(t),\tq_n(t))=(M_n, \nabla U(M_n))$.

Let $(\ty, \tq) \in C^1 \left( [0,t]; \ltord \right)$ be the solution of \eqref{Eq: HJsystemforU} with terminal data $(\ty(t),\tq(t))=(M_2, \nabla U(M_2))$. Note that $(\ty, \tq)$ is well defined since the existence of the solution for all times is guaranteed by the fact that the right-hand side of \eqref{Eq: HJsystemforU} is uniformly Lipschitz in $(\tx,\tp)$.

The solutions $(\ty_n, \tq_n)$ to \eqref{Eq: HJsystemforU} have terminal data $(\ty_n(t),\tq_n(t))$. This data converges, in the strong sense, to the terminal data $(\ty(t),\tq(t))$ of another solution $(\ty,\tq)$ of the same system. Therefore, by the stability of ODEs we obtain that $\ty_n(s) \rightarrow \ty(s), \tq_n(s) \rightarrow \tq(s)$, uniformly in the interval $[0,t]$. But this means that $\ty_n(0)=\tx(0) \circ G_n +Z_n$ converges strongly to $\ty(0)$, hence $\ty(0) \sim \tx(0)$. Then, by Lemma \ref{Lemma: DiffofPerInvfunction}, $\ty(0)$ is also a differentiability point of the function $U$. By the continuity of the gradient, we have
$$\tq(0)=\lim_{n \rightarrow \infty}{\tq_n(0)}=\lim_{n \rightarrow \infty}{\nabla U(\ty_n(0))}=\nabla U(\ty(0)).$$
This means that $\ty$ is a minimizing trajectory starting at the differentiability point $\ty(0)$. Therefore, $(M_2, \nabla U(M_2))=(\ty(t),\tq(t)) \in \DD_t$ or, equivalently, $M_2 \in D_t$.
\end{proof}
\begin{lemma}\label{Lemma: clsrDD_t} For every $t>0$ one has that $\overline{D_t}=\bigcap \limits_{s<t} D_s$.
\end{lemma}
\begin{proof} Suppose $M \in \overline{D_t}$. Then, there exist points $M_n \in D_t$ such that $M_n \rightarrow M$.
Because $M_n \in D_t$, $U$ is differentiable at $M_n$ and $(M_n,\nabla U(M_n)) \in \DD _t$. Furthermore, there exist minimizing trajectories $\tx_n$ such that $(\tx_n(t),\tp_n(t))=(M_n,\nabla U(M_n))$ and $(\tx_n(0),\tp_n(0)) \in \DD$. Since $\tx _n$ is a minimizer, the Lagrangian $L$ satisfies the assumption viii), and $U$ is a terminal cost function (as well as a value function), there exists a constant $C$ depending on time $t$ such that
\begin{equation}\label{Eq: SemConcatM_n}
U(M_n+H) \geq U(M_n)+\langle \nabla U(M_n), H \rangle - C \|H\|^2,
\end{equation}
for all $H \in L^2(I; \RR ^d)$ and all $n$. Because $U$ is Lipschitz, the sequence $\{\nabla U(M_n)\}$ is bounded. 
Consequently, it has a weakly convergent subsequence with a limit $P \in L^2(I; \RR ^d)$. Hence, by passing to the limit in \eqref{Eq: SemConcatM_n}, we obtain that
$$U(M+H) \geq U(M)+\langle  P, H \rangle - C \|H\|^2,$$
for all $H \in L^2(I; \RR ^d)$. Therefore, $P$ belongs to the subdifferential $D^{-}U(M)$. Due to the semiconcavity of $U$, the superdifferential $D^{+}U(M)$ is non-empty. Consequently, $U$ is differentiable at $M$ and $P=\nabla U(M)$ is its gradient.

By the continuity of the gradient, we have that $\nabla U(M_n) \rightarrow \nabla U(M)$ in the strong $L^2$ sense. Next, solve \eqref{Eq: HJsystemforU} with terminal data $\tx(t)=M, \tp(t)=\nabla U(M)$. We have that $(\tx_n(t),\tp_n(t)) \to (\tx(t),\tp(t))$. Hence, $\tx_n(s) \rightarrow \tx(s), \tp_n(s) \rightarrow \tp(s)$, uniformly in the interval $[0,t]$ by the stability of solutions of ODEs. Furthermore,
\begin{equation*} U(\tx_n(0))=\int_0^t{L_c(\tx_n,\dot{\tx}_n) +\bar{H}(c) ds}+U(\tx_n(t)).
\end{equation*}
Thus, by passing to the limit, we obtain
\begin{equation*} U(\tx(0))=\int_0^t{L_c(\tx,\dot{\tx})+\bar{H}(c) ds}+U(\tx(t)).
\end{equation*}
Consequently, $\tx$ is a minimizing trajectory. Since $U$ is differentiable at all points of the minimizing trajectory,  except, possibly, at  the starting point, we obtain that $(M,\nabla U(M)) \in \DD _s$, for all $s<t$. Therefore, $M \in \bigcap \limits_{s<t} D_s$, so $\overline{D_t} \subset \bigcap \limits_{s<t} D_s$.

Now, we claim that $\bigcap \limits_{s<t} D_s \subset \overline{D_t}$. Suppose $M \in \bigcap \limits_{s<t} D_s$ or, equivalently, $(M,\nabla U(M)) \in \bigcap \limits_{s<t} \DD_s$. Then, for every $0<s<t$ there exists a minimizing trajectory $\tx_s(\tau)$ such that $(\tx_s(s),\tp_s(s))=(M,\nabla U(M))$.
Let
$(\tx,\tp)$ be the solution of \eqref{Eq: HJsystemforU} with data $(\tx(t),\tp(t))=(M,\nabla U(M))$.
By the uniqueness of the solution to \eqref{Eq: HJsystemforU}, we have that $(\tx_s(\tau),\tp_s(\tau))=(\tx(\tau+t-s),\tp(\tau+t-s))$, for any $s<t$. Therefore, the trajectory $\tx$ is a minimizer for any starting point $(\tx(t-s),\tp(t-s))$. Consequently, $(\tx(2t-s),\tp(2t-s)) \in \DD _t$. So $\tx(2t-s) \in D_t$, but $\tx(t)=\lim_{s \rightarrow t}{\tx(2t-s)}$. Hence, $M=\tx(t) \in \overline{D_t}$.
\end{proof}
\begin{corollary}\label{crl: InvofclsrD_t} If $M_1 \in \overline{D_t}$ and $M_2 \sim M_1$, then $M_2 \in \overline{D_t}$.
\end{corollary}
\begin{proof}
The proof follows from Lemmas \ref{Lemma: InvofD_t} and \ref{Lemma: clsrDD_t}.
\end{proof}

\begin{corollary} For every $t>0$ one has that $\overline{\DD _t} = \bigcap \limits_{s<t} \DD _s$.
\end{corollary}
\begin{proof} For any $(M,P) \in \overline{\DD _t}$, we have that $P= \nabla U(M)$ and $M \in \overline{D_t}$. Thus, by Lemma \ref{Lemma: clsrDD_t}, $M \in \bigcap \limits_{s<t}D _s$. Hence $(M,P)=(M,\nabla U(M)) \in \bigcap \limits_{s<t} \DD _s$.

On the other hand, if $(M,P) \in \bigcap \limits_{s<t} \DD _s$, then $P=\nabla U(M)$ and $M \in \bigcap \limits_{s<t}D _s$. Therefore, from the Lemma \ref{Lemma: clsrDD_t}, we have that $M \in \overline{D_t}$. Consequently, there exist $M_n \in D_t$ such that $M_n \rightarrow M$. Because $M_n \in D_t$, the function $U$ is differentiable at $M_n$ and $(M_n, \nabla U(M_n)) \in \DD _t$. By the continuity of the gradient, $\nabla U(M_n) \rightarrow \nabla U(M)$ in the strong $L^2$ sense. Then $(M_n,\nabla U(M_n)) \rightarrow (M,\nabla U(M))$ and $(M,P)=(M,\nabla U(M)) \in \overline{\DD _t}$.
\end{proof}

\begin{lemma}\label{lma: Existenceof2sidedMins} The set $\bigcap \limits_{t>0}\DD _t$ is non-empty.
\end{lemma}
\begin{proof} Since all the sets $\DD _t$ are graphs, the statement in the lemma is equivalent to
\begin{equation*}\label{dinfty}
	D_{\infty}=\bigcap \limits_{t>0}D_t \neq \emptyset.
\end{equation*}	
From Lemma \ref{Lemma: clsrDD_t}, we obtain that $D_{\infty}=\bigcap \limits_{t>0}\overline{D_t}$. The sets $\{\overline{D_t}\}_{t>0}$ are closed nested sets. Consider projections of $A_t$ onto $\SS^d$ through the projection operator $\pi$ (see \eqref{proj}).

Due to Corollary \ref{crl: InvofclsrD_t}, the sets $\overline{D_t}$ contain only full equivalence classes with respect to the equivalence relation $\sim$. Since the sets $\overline{D_t}$ are closed, the sets $A_t$ are also closed. Additionally, they are compact, because $\SS^d$ is compact. Accordingly, they have a non-empty compact intersection
\begin{equation*}\label{Ainfty}
	A_{\infty}=\bigcap \limits_{t>0}A_t.
\end{equation*}

Therefore,
\begin{equation*}\label{dinfty_ainfty}
	D_{\infty}= \pi ^{-1}(A_{\infty}) \neq \emptyset.
\end{equation*}
\end{proof}


Now, we have all the prerequisites to prove our next main result, 
Theorem \ref{t1p3}.

\begin{proof}[Proof of Theorem \ref{t1p3}]
Consider the set
\begin{equation*}\label{Omega}
	\Omega=\left\{\left(M,-D_pH_c(M,P)\right)\ ;\ (M,P) \in  \bigcap \limits_{t>0}\DD _t\right\} \subset \mathcal{T}\ltord
\end{equation*}
Since $\bigcap \limits_{t>0}\DD _t \subset \mathcal{T}^*\ltord$ is invariant under the Hamiltonian flow, $\Omega$ is invariant under the Euler-Lagrange flow.

The projection operator $\pi$ is continuous, and the set $A_{\infty}$ is compact and hence closed. Therefore, the set $D_{\infty}$ is closed. Since the gradient of $U$ is continuous on the set of differentiability, we obtain that $\bigcap \limits_{t>0}\DD _t$ and $\Omega$ are also closed.
\end{proof}

From Lemma \ref{lma: Existenceof2sidedMins}, it is straightforward to prove the existence of two-sided minimizers of the Lax-Oleinik semigroup.
\begin{proof}[Proof of the Corollary \ref{cortwosidedmin}] For every point $M \in D_{\infty}$, consider the minimizing trajectory $\tx$ that passes through $M$.
\end{proof}
 The existence of \textit{weakly invariant} minimizing measures on the tangent space of $\Pp(\Tt^d)$ was shown in \cite{tudgangbo'12}. Subsequently, 
 in \cite{bessi'13}, the author established the existence of \textit{invariant} (or \textit{strongly invariant}) minimizing measures on $\mathcal{T}L^2\left([0,1]\right)$. 
Here, we settle the remaining question, namely the existence of strongly invariant measures in $\mathcal{T}\ltord$ for all $d\geq 1$.
For that, 
we extend the methods introduced by Fathi in \cite{F} to the infinite-dimensional setting.

Let $\mathfrak{A}$ be the Borel $\sigma$-algebra of the subsets of $A_{\infty}$. 
Define
\begin{equation*}\label{Dsigmalgebra}
	\mathfrak{D}=\pi^{-1}(\mathfrak{A})=\{\pi^{-1}(A)\ ; \ A\subset A_{\infty}\}.
\end{equation*}
Note that $\mathfrak{D}$ is a $\sigma$-algebra of subsets of $D_{\infty}$. For every $C \in \mathfrak{D}$, set
\[B_C=\left\{(M,-D_pH_c(M,\nabla U(M)))\ ;\ M \in C \right\}
\]
and consider
\begin{equation}\label{Bsigmalgebra}
	\mathfrak{B}=\left\{B_C\ ;\ C \in \mathfrak{D}\right\}.
\end{equation}
$\mathfrak{B}$ is a $\sigma$-algebra of subsets of $\Omega$.
\begin{lemma}[Riesz Representation Theorem and Compactness]\label{Riezs}
\begin{itemize} We have that
	\item[i.] for every linear bounded functional $Q$ acting on the set of continuous rearrangement-invariant functions $F:\Omega \to \RR$, there exists a  measure $\mu$ on $(\Omega, \mathfrak{B})$ such that
	\begin{equation*}
	Q(F)=\int\limits_{\Omega}Fd\mu;
	\end{equation*}
	\item[ii.] the space of probability measures on $(\Omega, \mathfrak{B})$ 
 	 is a narrowly compact space.
\end{itemize}

\end{lemma}
\begin{proof}
	\begin{itemize}
		\item[i.] Let $f:A_{\infty}\to \RR$ be a continuous function. For $(M,N) \in \Omega$, let $F_f(M,N)=F_f(M,-D_pH_c(M,\nabla U(M)))=f(\pi(M))$. Then, $F_f$ is rearrangement invariant and continuous. Define
		\begin{equation*}
		q(f)=Q(F_f). 
		\end{equation*}
		$q$ is a linear bounded functional acting on continuous functions $f:A_{\infty}\to \RR$. Since $A_{\infty}$ is compact, the Riesz Representation Theorem yields the existence of a Borel measure $\nu$ on $A_{\infty}$ such that
		\begin{equation*}
		q(f)=\int\limits_{A_{\infty}}fd\nu.
		\end{equation*}
		Consider the measure
		\begin{equation}\label{mu_nu}
		\mu(B_C)=\nu(\pi(C)),\quad C \in D_{\infty}.
		\end{equation}
		Then, $\mu$ is a measure on $(\Omega, \mathfrak{B})$. Moreover, for every rearrangement invariant and continuous $F:\Omega \to \RR$, we have
		\begin{equation*}
		\int\limits_{\Omega}F d\mu=\int\limits_{A_{\infty}}fd\nu,
		\end{equation*}
		where $f(a)=F(M,-D_pH_c(M,\nabla U(M))$, for some $M \in \pi^{-1}(a)$ and all $a\in A_{\infty}$. Note that $F_f=F$. Hence, we have that
		\[\int\limits_{A_{\infty}}fd\nu=q(f)=Q(F_f)=Q(F).
		\]
		Consequently, the previous two identities give

\[
Q(F)=\int\limits_{\Omega}F(M,N) d\mu.
\]
		\item[ii.] Suppose $\{\mu_n\}$ are probability measures on $(\Omega, \mathfrak{B})$. Consider the sequence of measures $\nu_n$ given by \eqref{mu_nu}. Since, $\nu_n$ are supported on a compact set $A_{\infty}$, they form a narrowly precompact sequence. Hence, there exists a measure $\nu_{\infty}$ such that $\nu_n \to \nu_{\infty}$ narrowly. Now, define $\mu_{\infty}$ via \eqref{mu_nu}. Then, it is straightforward to verify that $\mu_n$ converges to $\mu_{\infty}$ narrowly (tested against continuous rearrangement-invariant functions).  
	\end{itemize}
\end{proof}
Let $\tx$ be a solution of
the Euler-Lagrange equation \eqref{eq: InfDynamics}. Define
\[\Phi(t;M,N)=(\Phi^1(t;M,N), \Phi^2(t;M,N))=(\tx(t),\dot{\tx}(t)).
\]
\begin{definition}
	A measure $\mu$ defined on $(\Omega, \mathfrak{B})$  is  invariant under the Euler-Lagrange flow (invariant, for brevity), if for every $t>0$
	\begin{equation*}\label{invmeasuredef}
	\mu=\Phi(t ;\cdot , \cdot) \# \mu.
	\end{equation*}
\end{definition}
	\begin{proof}[Proof of the Corollary \ref{minmeasures}]
		The finite-dimensional analog of this statement is presented in \cite{F} (Corollary 4.4.9). The proof extends to the
		infinite-dimensional setting without significant difficulties. Hence, here, we give the main components of the proof.
		
		The proof of the inequality
		\[\int\limits_{\Omega}L_cd\mu \geq -\bar{H}(c)
		\]
		for invariant measures $\mu$, extends to the infinite-dimensional setting with no changes. Next, we prove the existence of minimizing measures. Fix a point $M_0 \in \Omega$. Consider the linear bounded functional
		\[F \mapsto \frac{1}{t}\int\limits_{0}^{t}F(\Phi(t;M_0,N_0))dt,
		\]
		where $N_0=-D_pH_c(M_0,\nabla U(M_0))$. From Lemma \ref{Riezs}, we have that there exist  measures $\mu_t$ on $\mathfrak{B}$ such that
		\[\int 
		\limits_{\Omega} Fd\mu_t=\frac{1}{t}\int\limits_{0}^{t}F(\Phi(s;M_0,N_0))ds
		\]
		for all continuous rearrangement-invariant functions $F:\Omega \to \RR$. The family of measures $\{\mu_t\}$ is precompact due to Lemma \ref{Riezs}. Hence, there exists a measure $\mu_{\infty}$ and a sequence $t_n \to \infty$ such that $\mu_{t_n} \to \mu_{\infty}$ narrowly. In other words,
		\[\int \limits_{\Omega} Fd\mu_{\infty}=\lim\limits_{n \to \infty} \int \limits_{\Omega} Fd\mu_{t_n}=\lim\limits_{n \to \infty} \frac{1}{t_n}\int\limits_{0}^{t_n}F(\Phi(s;M_0,N_0))ds,
		\]
		for continuous rearrangement-invariant test functions $F$. Let $h>0$, we have
		\begin{align*}
		\int\limits_{\Omega}F(\Phi(h;M,N))d\mu_{\infty}&=\lim \limits_{n\to \infty}\frac{1}{t_n}\int\limits_{0}^{t_n} F(\Phi(h;\Phi(s;M_0,N_0)))ds\\
		&=\lim \limits_{n\to \infty}\frac{1}{t_n}\int\limits_{0}^{t_n} F(\Phi(s+h;M_0,N_0))ds\\
		&=\lim \limits_{n\to \infty}\frac{1}{t_n}\int\limits_{h}^{t_n+h} F(\Phi(s;M_0,N_0))ds\\
		&=\lim \limits_{n\to \infty}\frac{1}{t_n}\int\limits_{0}^{t_n} F(\Phi(s;M_0,N_0))ds\\
		&=\int\limits_{\Omega}F(M,N)d\mu_{\infty}.
		\end{align*}
		Therefore, $\mu_{\infty}$ is invariant under the Euler-Lagrange flow.
		
		Furthermore, set $F=L_c$ and observe that
		\begin{align*}
		\int\limits_{\Omega} L_cd\mu_{\infty}&=\lim \limits_{n\to \infty}\frac{1}{t_n}\int\limits_{0}^{t_n} L_c(\Phi(s;M_0,N_0))ds\\
		&=\frac{U(M_0)-U(\Phi^1(t_n;M_0,N_0))-t_n\bar{H}(c)}{t_n}\\
		&=-\bar{H}(c).
		\end{align*}
	The previous identity gives that $\mu_{\infty}$ is a minimizing measure.
	\end{proof}

\def\cprime{$'$} \def\cprime{$'$}


\end{document}